\documentclass[11pt]{amsart}

\usepackage{amssymb}
\usepackage{hyperref}
\usepackage{amsmath,amscd}
\usepackage[margin=1in]{geometry}

\newtheorem{theorem}{Theorem}[section]

\newtheorem{proposition}[theorem]{Proposition}
\newtheorem{corollary}[theorem]{Corollary}

\theoremstyle{definition}

\newtheorem{example}[theorem]{Example}

\theoremstyle{remark}
\newtheorem{remark}[theorem]{Remark}

\def\qand{\quad\text{and}\quad}

\begin{document}

\title{Compositions with restricted parts} 
\author{Jia Huang}
\address{Department of Mathematics and Statistics, University of Nebraska at Kearney, NE 68849, USA}
\email{huangj2@unk.edu}

\begin{abstract}
Euler showed that the number of partitions of $n$ into distinct parts equals the number of partitions of $n$ into odd parts. This theorem was generalized by Glaisher and further by Franklin. Recently, Beck made three conjectures on partitions with restricted parts, which were confirmed analytically by Andrews and Chern and combinatorially by Yang.

Analogous to Euler's partition theorem, it is known that the number of compositions of $n$ with odd parts equals the number of compositions of $n+1$ with parts greater than one, as both numbers equal the Fibonacci number $F_n$. Recently, Sills provided a bijective proof for this result using binary sequences, and Munagi proved a generalization similar to Glaisher's result using the zigzag graphs of compositions. Extending Sills' bijection, we obtain a further generalization which is analogous to Franklin's result. We establish, both analytically and combinatorially, two closed formulas for the number of compositions with restricted parts appearing in our generalization. We also prove some composition analogues for the conjectures of Beck.
\end{abstract}



\maketitle

\section{Introduction}

Partitions and compositions are prevalent in enumerative combinatorics and also play important roles in many other fields, such as the symmetric function theory, the representation theory of symmetric groups and Hecke algebras, combinatorial Hopf algebras, etc.
See, for example, Andrews--Eriksson~\cite{AndrewsEriksson}, Grinberg--Reiner~\cite{GrinbergReiner}, and Heubach--Mansour~\cite{HeubachMansour}.

Using generating functions, Euler proved the following well-known theorem concerning partitions with restricted parts.

\begin{theorem}[Euler] 
The number of partitions of $n$ into distinct parts equals the number of partitions of $n$ into odd parts.
\end{theorem}

Glaisher generalized Euler's partition theorem to the result below, which specializes to Euler's theorem when $k=2$.

\begin{theorem}[Glaisher] 
Given an integer $k\ge1$, the number of partitions of $n$ with no part occurring $k$ or more times equals the number of partitions of $n$ with no parts divisible by $k$.
\end{theorem}

Franklin obtained a further generalization of Euler's partition theorem, which recovers the result of Glaisher when $m=0$.

\begin{theorem}[Franklin] 
Given integers $k\ge1$ and $m\ge0$, the number of partitions of $n$ with $m$ distinct parts each occurring $k$ or more times equals the number of partitions of $n$ with exactly $m$ distinct parts divisible by $k$.
\end{theorem}

Recently, Beck made three conjectures on partitions with restricted parts in the On-Line Encyclopedia of Integer Sequence~\cite{OEIS}.
Andrews~\cite{Andrews} and Chern~\cite{Chern} proved the conjectures of Beck using generating functions, and Yang~\cite{Yang} proved these conjectures using Glaisher's bijection.
In general, it seems difficult to obtain closed formulas for the number of partitions, whether with or without part constraints; see, e.g., Sills~\cite{Sills}.

The main theme of this paper is to study analogues of the above theorems for compositions instead of partitions. 
Unlike in the case of partitions, we are able to obtain closed formulas for compositions with restricted parts, with both analytic and combinatorial proofs.
We also explore analogues of Beck's conjectures in the setting of compositions.
Generally speaking, partitions attract much more attention than compositions, but there have been some recent efforts on finding composition analogues of partition identities, such as the work of Munagi~\cite{Munagi}, Munagi--Sellers~\cite{MunagiSellers}, and Sills~\cite{Sills}.
It is certainly our hope that this paper, together with the above cited references, can bring more attention to the study of compositions.
In fact, after the first version of this paper was uploaded to the arXiv, we were aware via personal communication that it provided motivation to new work of Li and Wang~\cite{LiWang}, which includes a different bijective proof for one of our main results (Theorem~\ref{thm:comp3}) and more composition analogues of Beck's conjectures.

To summarize our new results, we first state a known composition analogue of Euler's partition theorem.

\begin{theorem}\label{thm:comp1}
The number of compositions of $n$ with odd parts equals the number of compositions of $n+1$ with parts greater than one.
\end{theorem}

Both numbers in Theorem~\ref{thm:comp1} are equal to the \emph{Fibonacci number} $F_n$ defined by $F_n:=F_{n-1}+F_{n-2}$ with $F_0=0$ and $F_1=1$; see, e.g., Cayley~\cite{Cayley}, Grimaldi~\cite{Grimaldi}, and Stanley~\cite[Exercise 1.35]{EC1}.
Recently, Sills~\cite{Sills} provided a bijective proof of Theorem~\ref{thm:comp1} using the binary sequence encoding of compositions.

One can also represent a composition as a \emph{zigzag graph} or equivalently, a \emph{ribbon diagram}; this is similar to the well-known Ferrers/Young diagram of a partition.
Using the zigzag graphs of compositions, Munagi~\cite[Theorem 1.2]{Munagi} generalized Theorem~\ref{thm:comp1} to the following result.

\begin{theorem}\label{thm:comp2}
For any integer $k\ge1$, the number of compositions of $n$ with parts congruent to $1$ modulo $k$ equals the number of compositions of $n+k-1$ with parts no less than $k$.
\end{theorem}

Theorem~\ref{thm:comp2} generalizes Theorem~\ref{thm:comp1} similarly as Glaisher's theorem generalizes Euler's partition theorem.
The two equal numbers in Theorem~\ref{thm:comp2} both appear in OEIS~\cite[A003269 for $k=4$]{OEIS}.
The first number in Theorem~\ref{thm:comp2}, denoted by $a_{k,n}$, has a simple closed formula
\begin{equation}\label{eq:Dani}
a_{k,n} = \sum_{0\le j\le (n-1)/k} \binom{n-1-j(k-1)}{j}
\end{equation}
by Dani~\cite{Dani} and Munagi~\cite{Munagi}.
The generating function $A_k(x):=\sum_{n\ge0} a_{k,n} x^n$ can be derived from a more general result of Heubach and Mansour~\cite[Theorem~3.13]{HeubachMansour}.

For $k\ge2$, the number $a_{k,n}$ of compositions of $n$ with all parts congruent to $1$ modulo $k$ also equals the number of compositions of $n-1$ with all parts equal to $1$ or $k$.
One can prove this bijectively by replacing each part congruent to $1$ modulo $k$ with a string of $k$'s followed by a $1$ and striking out the last $1$.
See also Munagi~\cite[Theorem 1.2]{Munagi}.
These two equal numbers both appear in OEIS~\cite[A005710 for $k=8$]{OEIS}.
The latter number was studied by Chinn and Heubach~\cite{ChinnHeubach} and their result~\cite[Lemma 1]{ChinnHeubach} coincides with the generating function $A_k(x)$ upon a shift of terms.

We provide a proof for Theorem~\ref{thm:comp2} based on the bijective proof of Theorem~\ref{thm:comp1} by Sills~\cite{Sills}.
Although it gives the same bijection as the proof of Munagi~\cite[Theorem 1.2]{Munagi}, we can further extend it to establish the following result, which generalizes Theorem~\ref{thm:comp2} similarly as Franklin's theorem generalizes Glaisher's theorem.

\begin{theorem}\label{thm:comp3}
For any integers $k\ge1$ and $m\ge0$, the number of compositions of $n$ with exactly $m$ parts not congruent to $1$ modulo $k$, each of which is greater than $k$, equals the number of compositions of $n+k-1$ with exactly $m$ parts less than $k$, each of which is preceded by a part at least $k$ and followed by either the last part or a part greater than $k$.
\end{theorem}

Let $a_{k,n}^{(m)}$ denote the number in Theorem~\ref{thm:comp3}.
For $k\ge3$ or $m\ge2$ we do not see the sequence $a_{k,n}^{(m)}$ in OEIS.
When $k=2$, $m=1$, and $n\ge1$, this sequence appears in OEIS~\cite[A029907]{OEIS} 
with some interesting combinatorial interpretations, which are different from what we have in Theorem~\ref{thm:comp3}.
This sequence is also related to the composition analogues we obtain for Beck's conjectures in Section~\ref{sec:Beck} and one can find more details there.

We establish two closed formulas for the number $a_{k,n}^{(m)}$, which specialize to the formula~\eqref{eq:Dani} for the number appearing in Theorem~\ref{thm:comp2} when $m=0$.

\begin{theorem}\label{thm:formula}
For $m,n\ge0$ and $k\ge2$ we have
\begin{align*}
a_{k,n}^{(m)} 
&= \sum_{ \substack{ \lambda\subseteq (k-2)^m \\ i+(k+1)m+jk+ |\lambda| = n }} \binom{i}{m} \binom{i+j-1}{j} m_\lambda(1^m) \\
&= \sum_{i+(k+1)m+jk+\ell(k-1)+h=n} (-1)^\ell \binom{i}{m} \binom{i+j-1}{j} \binom{m}{\ell} \binom{m+h-1}{h}.
\end{align*}
\end{theorem}
Here $\lambda\subseteq(k-2)^m$ means that $\lambda$ is a partition with at most $m$ parts, each no more than $k-2$, $|\lambda|$ is the sum of all the parts of $\lambda$, and $m_\lambda(1^m)$ is the specialization of the monomial symmetric function indexed by $\lambda$ evaluated at the vector $(1,\ldots,1)$ of length $m$.

The first formula in Theorem~\ref{thm:formula} is a positive summation.
The second formula is somewaht simpler, but carries negative signs.
We provide two proofs of Theorem~\ref{thm:formula}.
One is analytic, using the generating function technique, while the other is purely combinatorial, with the alternating signs in the second formula explained by inclusion-exclusion. 

The paper is structured as follows.
We first provide some preliminaries on partitions and compositions in Section~\ref{sec:prelim}.
Then we prove Theorem~\ref{thm:comp2} in Section~\ref{sec:comp2} using Sills' bijection, and generalize it to Theorem~\ref{thm:comp3} in Section~\ref{sec:comp3}.
We next show Theorem~\ref{thm:formula} both analytically and combinatorially in Section~\ref{sec:formula}.
Finally we summarize recent studies on three conjectures of Beck for partitions with restricted parts, and prove some composition analogues in Section~\ref{sec:Beck}.

\section{Preliminaries}\label{sec:prelim}

Given integers $a$ and $b$, we define the binomial coefficient
\[ \binom{a}{b} := 
\begin{cases}
\frac{a!}{b!(a-b)!}, & \text{if } a\ge b\ge 1, \\
1, & \text{if } b=0, \\
0, & \text{otherwise}.
\end{cases} \]
For any integer $d\ge1$, it is easy to show the following identity, which will be used in the analytic proof of Theorem~\ref{thm:formula}:
\begin{equation}\label{eq:NegPower}
\frac{1}{(1-x)^d} = \sum_{i=0}^\infty \binom{i+d-1}{i} x^i.
\end{equation}

A \emph{partition of $n$} is a weakly decreasing sequence $\lambda=(\lambda_1,\ldots,\lambda_\ell)$ of positive integers with sum $\lambda_1+\cdots+\lambda_\ell=n$; it is common to use the symbol $\lambda\vdash n$ to denote this.
The \emph{size} of $\lambda$ is $|\lambda|:=n$, the \emph{parts} of $\lambda$ are the integers $\lambda_1,\ldots,\lambda_\ell$, and the \emph{length} of $\lambda$ is the number of parts $\ell(\lambda):=\ell$.

Let $\lambda \subseteq r^d$ denote that $\lambda=(\lambda_1,\ldots,\lambda_\ell)$ is a partition with at most $d$ parts, each part no more than $r$, i.e., $\lambda_1\le r$ and $\ell\le d$.
For $i=0,1,\ldots,r$, let $m_i$ be the number of parts of the partition $\lambda\subseteq r^d$ that are equal to $i$.
Then $m_0+m_1+\cdots+m_r=\ell \le d$ and $m_1+2m_2+\cdots+rm_r=|\lambda|$.
The \emph{monomial symmetric function} $m_\lambda(x_1,\ldots,x_d)$ is the sum of the monomials $x_1^{a_1}\cdots x_d^{a_d}$ for all rearrangements $(a_1,\ldots,a_d)$ of $(\lambda_1,\ldots,\lambda_d)$, where $\lambda_{\ell+1}=\cdots=\lambda_d=0$.
The evaluation of $m_\lambda(x_1,\ldots,x_d)$ at the vector $(1,\ldots,1)$ of length $d$ is
\[ m_\lambda(1^d) = \binom{m}{m_0,\ldots,m_r} = \frac{m!}{m_0!\cdots m_r!}. \]
One sees that
\begin{align}
(1+x+x^2+\cdots+x^r)^d 
&= \sum_{n\ge0} \sum_{ \substack{ m_0+m_1+\cdots+m_r=d \\ m_1+2m_2+\cdots+rm_r=n } } \binom{m}{m_0,\ldots,m_r} x^n \\
&= \sum_{n\ge0} \sum_{ \substack{ \lambda\subseteq r^d \\ |\lambda|=n } } m_\lambda(1^d) x^n. \label{eq:monomial}
\end{align}
We will use the identity~\eqref{eq:monomial} in our analytic proof of Theorem~\ref{thm:formula}. 

Next, a \emph{composition of $n$} is a sequence $\alpha=(\alpha_1,\ldots,\alpha_\ell)$ of positive integers with $\alpha_1+\cdots+\alpha_\ell=n$; we use the symbol $\alpha\models n$ to denote this.
The \emph{parts} of $\alpha$ are $\alpha_1,\ldots,\alpha_\ell$, which are not necessarily decreasing.
The \emph{length} of $\alpha$ is the number of parts $\ell(\alpha):=\ell$.
We say that a part $\alpha_i$ is \emph{preceded by} the part $\alpha_{i-1}$ if $i>1$, and \emph{followed by} the part $\alpha_{i+1}$ if $i<\ell$.
The \emph{descent set} of the composition $\alpha$ is 
\[ D(\alpha) := \{\alpha_1,\alpha_1+\alpha_2,\ldots,\alpha_1+\cdots+\alpha_{\ell-1}\}.\]
The map $\alpha\mapsto D(\alpha)$ is a bijection from compositions of $n$ to subsets of $[n-1]:=\{1,2,\ldots n-1\}$.
Furthermore, a subset $S\subseteq [n-1]$ can be encoded by a binary sequence of length $n-1$ whose $i$th component is $1$ if $i\in S$ or $0$ otherwise.
Therefore we have a bijection between compositions of $n$ and binary sequences of length $n-1$.
For example, the composition $\alpha=(1,7,1,4)\models 13$ has descent set $D(\alpha)=\{1,8,9\}\subseteq[12]$ and corresponds to the binary sequence $100000011000$.

Finally, the \emph{opposite} of a binary sequence $\mathbf{b}$ is the equally long binary sequence whose  $i$th component is different from the $i$th component of $\mathbf{b}$ for all $i$.
For example, the opposite of $110001011$ is $001110100$.

\section{Proof of Theorem~\ref{thm:comp2} using Sills' bijection}\label{sec:comp2}
Munagi~\cite[Theorem 1.2]{Munagi} bijectively proved Theorem~\ref{thm:comp2} using the zigzag graphs of compositions.
Now we provide a proof based on the bijective proof of Theorem~\ref{thm:comp1} due to Sills~\cite{Sills}.
The bijection constructed in our proof agrees with the bijection given by the proof of Munagi~\cite[Theorem 1.2]{Munagi}, but we can further extend it to prove Theorem~\ref{thm:comp3} in Section~\ref{sec:comp3}.

\begingroup
\def\thetheorem{\ref{thm:comp2}}
\begin{theorem}
For any integer $k\ge1$, the number of compositions of $n$ with all parts congruent to $1$ modulo $k$ equals the number of compositions of $n+k-1$ with no parts less than $k$.
\end{theorem}
\addtocounter{theorem}{-1}
\endgroup

\begin{proof}
Let $\alpha=(\alpha_1,\ldots,\alpha_\ell)$ be a composition of $n$, which corresponds to a binary sequence $\mathbf{b}$ of length $n-1$.
Let $\mathbf{c}$ denote the opposite binary sequence of $\mathbf{b}$.
Assume $a_i\equiv 1\pmod k$ for all $i=1,2,\ldots,\ell$.
This implies that the zeros in $\mathbf{b}$ [or the ones in $\mathbf{c}$, resp.] appear in strings of length divisible by $k$.
Thus we can replace each maximal substring of ones in $\mathbf{c}$ with an equally long string of the following form:
\begin{equation}\label{eq:01}
\underbrace{0\cdots0}_{k-1} 1 \underbrace{0\cdots0}_{k-1} 1 \cdots \underbrace{0\cdots0}_{k-1} 1 
\end{equation}
The resulting binary sequence $\widetilde{\mathbf{c}}$ corresponds to the descent set of a composition $\widetilde{\alpha}$ of $n$ whose parts are all at least $k$ except the last one.
Adding $k-1$ to the last part of $\widetilde\alpha$ gives a composition of $n+k-1$ with no parts less than $k$.

Conversely, given a composition of $n+k-1$ with no parts less than $k$, one can subtract $k-1$ from the last part and get a composition of $n$, which corresponds to a binary sequence of length $n-1$.
Every one in this binary sequence is preceded by at least $k-1$ zeros, and thus replacing each substring of the form $0^{k-1}1$ with $1^k$ gives a binary sequence of length $n-1$ with ones appearing in strings of length divisible by $k$.
Then the opposite sequence has zeros appearing in strings of length divisible by $k$ and corresponds to a composition of $n$ with all parts congruent to $1$ modulo $k$.
\end{proof}

\begin{example}
The composition $\alpha=(1,7,1,4)\models 13$ has all parts congruent to $1$ modulo $k=3$.
It corresponds to the binary string $\mathbf{b} = 100000011000$, whose opposite is $\mathbf{c} = 011111100111$. 
Replacing each maximal substring of ones in $\mathbf{c}$ with an equally long string of the form~\eqref{eq:01} gives the binary sequence $\widetilde{\mathbf{c}} = 000100100001$, which corresponds to the composition $\widetilde\alpha =(4,3,5,1)\models 13$.
Adding $k-1$ to the last part of $\widetilde\alpha$ gives the composition $(4,3,5,3)\models 15$ with no part less than $k$.

Conversely, the composition $(4,3,5,3)\models 15$ has no part less than $k=3$.
Substracting $k-1$ from its last part gives the composition $(4,3,5,1)\models 13$, which corresponds to a binary sequence $000100100001$.
Replacing each copy of $001$ with $111$ in this binary sequence gives the binary sequence $011111100111$, in which ones occur in strings of length divisible by $k$.
The opposite binary sequence $100000011000$ corresponds to the composition $(1,7,1,4)\models 13$ with all parts congruent to $1$ modulo $k$.
\end{example}

For $n\ge0$ let $a_{k,n}$ denote the number of compositions of $n$ with all parts congruent to $1$ modulo $k$; 
in particular, $a_{k,0}=1$ since it is vacuously true that any part of the empty composition is congruent to $1$ modulo $k$.
Dani~\cite{Dani} and Munagi~\cite{Munagi} obtained a closed formula for the number $a_{k,n}$ from its generating function
\[ A_k(x) := \sum_{n\ge0} a_{k,n} x^n. \]
One can derive $A_k(x)$ from a more general result of Heubach and Mansour~\cite[Theorem~3.13]{HeubachMansour}.
We give the formulas for $A_k(x)$ and $a_{k,n}$ in the following proposition and include a proof here to help the reader understand our proof of the more general Theorem~\ref{thm:formula} in Section~\ref{sec:formula}.

\begin{proposition}\label{prop:Ak}
For $n,k\ge1$ we have
\[ A_k(x) = \frac{1-x^k}{1-x-x^k} \qand
a_{k,n} = \sum_{0\le j\le (n-1)/k} \binom{n-1-j(k-1)}{j}. \]
\end{proposition}

\begin{proof}
Each composition of $n$ with all parts congruent to $1$ modulo $k$ must begin with a part of the form $ik+1$ for some integer $i\ge0$.
Removing the first part from this composition gives a composition of $n-ik-1$ with all parts congruent to $1$ modulo $k$. 
Thus we have 
\[ a_{k,n} = \sum_{i\ge0} a_{k,n-ik-1}, \qquad\forall n\ge1. \]
where $a_{k,n}:=0$ for $n<0$.
This recurrence relation implies that
\begin{align*}
A_k(x) 
&= 1+\sum_{i\ge0} x^{ik+1} A_k(x) 
= 1+\frac{x}{1-x^k} A_k(x).
\end{align*}
It follows that
\begin{align*}
A_k(x) 
&=1+\frac{x}{1-x-x^k} \\
&=1+ \sum_{i\ge0} x(x+x^k)^i \\
&=1+ \sum_{i\ge0} x^{i+1} \sum_{0\le j\le i} \binom{i}{j} x^{j(k-1)} \\
&=1+ \sum_{j\ge0} \sum_{i\ge j} \binom{i}{j} x^{j(k-1)+i+1}.
\end{align*}
Taking the coefficient of $x^n$ in the above series gives the desired formula for $a_{k,n}$.
\end{proof}

\section{Proof of Theorem~\ref{thm:comp3}}\label{sec:comp3}

In this section we further generalize Theorem~\ref{thm:comp2} to Theorem~\ref{thm:comp3}, which is restated below. 
For an example of the bijection in our proof, see Example~\ref{example}.
\begingroup
\def\thetheorem{\ref{thm:comp3}}
\begin{theorem}
For any integers $k\ge1$ and $m\ge0$, the number of compositions of $n$ with exactly $m$ parts not congruent to $1$ modulo $k$, each of which is greater than $k$, equals the number of compositions of $n+k-1$ with exactly $m$ parts less than $k$, each of which is preceded by a part at least $k$ and followed by either the last part or a part greater than $k$.
\end{theorem}
\addtocounter{theorem}{-1}
\endgroup

\begin{proof}
(i) Let $\alpha=(\alpha_1,\ldots,\alpha_\ell)$ be a composition of $n$, which corresponds to a binary sequence $\mathbf{b}$ of length $n-1$.
Let $\mathbf{c}$ be the opposite of $\mathbf{b}$.
One sees that every part $\alpha_i>1$ of $\alpha$ corresponds to a maximal string of ones in $\mathbf{c}$ whose length is $\alpha_i-1$, and we replace this maximal string of ones with an equally long string of the form 
\begin{equation}\label{eq:001}
\underbrace{0\cdots0}_{k-1} 1 \underbrace{0\cdots0}_{k-1} 1 \cdots \underbrace{0\cdots0}_{k-1} 1  \underbrace{0\cdots0}_{r-1} 1 
\end{equation}
where $r$ is the least positive residue of $\alpha_i-1$ modulo $k$.
The resulting binary sequence $\widetilde{\mathbf{c}}$ corresponds to a composition $\widetilde\alpha\models n$.

Assume that $\alpha$ has exactly $m$ parts not congruent to $1$ modulo $k$, all of which are greater than $k$.
If $\alpha_i\equiv1\pmod k$ then the parts of $\widetilde\alpha$ coming from the above string~\eqref{eq:001} are all at least $k$.
If $\alpha_i\not\equiv1\pmod k$ then the above string~\eqref{eq:001} gives exactly one part of $\widetilde\alpha$ that is less than $k$.
This part is preceded by a part that is at least $k$ if $\alpha_i>k$.
In addition, this part is followed by either the last part of $\widetilde\alpha$ or a part greater than $k$, since there is a $0$ right after the string of ones corresponding to $\alpha_i$ in the binary sequence $\mathbf{c}$ unless $i=\ell$.

The only part that we have not considered yet is the last part of $\widetilde\alpha$, which is possibly less than $k$. 
But adding $k-1$ to it gives a composition of $n+k-1$ whose last part now is at least $k$.
This composition has exactly $m$ parts less than $k$, each of which is preceded by a part at least $k$ and followed by either the last part or a part greater than $k$.

\vskip5pt\noindent
(ii) Conversely, let $\beta=(\beta_1,\ldots,\beta_\ell)$ be a composition of $n+k-1$, which corresponds to the binary sequence
\[ \left( 0^{\beta_1-1}1 0^{\beta_2-1}1 \cdots 0^{\beta_{\ell-1}-1}1 0^{\beta_\ell-1} \right).\]
Assume that $\beta$ has exactly $m$ parts less than $k$, each of which is preceded by a part at least $k$ and followed by either the last part or a part greater than $k$.
This implies $\beta_\ell\ge k$, so we can replace $0^{\beta_\ell-1}$ with $0^{\beta_\ell-k}$.
For each $i\in[\ell-1]$ we replace the string $0^{\beta_i-1}1$ with $0^{\beta_i-k} 1^k$ if $\beta_i\ge k$ or $1^{\beta_i}$ otherwise.
The opposite of the resulting binary sequence can be written as $\mathbf{b}=b_1 \cdots b_\ell$ where 
\[ b_i := 
\begin{cases}
1^{\beta_i-k}0^{k}, & \text{if } \beta_i\ge k, \\
0^{\beta_i}, & \text{if } \beta_i<k
\end{cases} \]
for all $i\in[\ell-1]$ and $b_\ell:=1^{\beta_\ell-k}$.
The binary sequence $\mathbf{b}$ corresponds to a composition $\beta' \models n$. 

There are exactly $m$ parts of $\beta'$ that are not congruent to $1$ modulo $k$, since they all come from the parts of $\beta$ less than $k$.
Recall that each part $\beta_i<k$ of $\beta$ is preceded by a part at least $k$ and followed by either the last part or a part greater than $k$.
Hence $b_i=0^{\beta_i}$ is preceded by a maximal string of zeros whose length is a positive multiple of $k$, and followed by either nothing or a $1$.
This gives a part of $\beta'$ not congruent to $1$ modulo $k$, and it must be greater than $k$.
\end{proof}

\begin{example}\label{example}
(i) The composition $\alpha=(5,4,6,1)\models 16$ has $m=2$ parts not congruent to $1$ modulo $k=3$, each greater than $k$.
It corresponds to the binary sequence $\mathbf{b}=000010001000001$ whose opposite is $\mathbf{c} = 111101110111110$.
Replacing each maximal substring of ones with a string of the form~\eqref{eq:001} gives the binary sequence $\widetilde{\mathbf{c}} = 001100010001010$, which corresponds to the composition $\widetilde{\alpha} = (3,1,4,4,2,2)\models 16$.
The part $5\not\equiv1\pmod k$ of $\alpha$ corresponds to the part $1<k$ in $\widetilde\alpha$, which is preceded by a $3\ge k$ and followed by a $4>k$.
The part $6\not\equiv1\pmod k$ of $\alpha$ corresponds to the first occurrence of $2$ in $\widetilde\alpha$, which is preceded by a $4\ge k$ and followed by the last part of $\widetilde\alpha$.
Adding $k-1$ to the last part of $\widetilde\alpha$ gives the composition $(3,1,4,4,2,4)\models 18$ with exactly $m$ parts less than $k$ as mentioned above since its last part is now $4\ge k$.

\vskip5pt\noindent
(ii) Conversely, the composition $\beta=(3,1,4,4,2,4)\models 18$ has $m=2$ parts less than $k=3$, which are $1$ and $2$, each preceded by a part at least $k$ and followed by either the last part or a part greater than $k$.
The composition $\beta$ corresponds to the binary sequence $00110001000101000$.
By the construction in the above proof, we have $\mathbf{b} = 000010001000001$, which corresponds to the composition $\beta'=(5,4,6,1)\models 16$.
There are exactly $m$ parts of $\beta'$ not congruent to $1$ modulo $k$, which are $5>k$ and $6>k$, coming from the two parts $1<k$ and $2<k$ of $\beta$.
\end{example}

\begin{remark}
After the first version of this paper was submitted to the arXiv, Li and Wang~\cite{LiWang} quickly found a simple bijective proof for Theorem~\ref{thm:comp3}, which is different from our proof.
\end{remark}

\begin{remark}
From computations we cannot find any connection between the number of compositions of $n$ with exactly $m$ parts not congruent to $1$ modulo $k$ and the number of compositions of $n+k-1$ with exactly $m$ parts less than $k$.
We give an example below to illustrate how our proof of Theorem~\ref{thm:comp3} would fail in this situation.
The composition $\alpha=(4,3,4)\models 11$ has exactly $m=1$ part not congruent to $1$ modulo $k=3$.
It corresponds to the binary sequence $\mathbf{b}=0001001000$ whose opposite is $\mathbf{c}=1110110111$.
Replacing each maximal string of ones with an equally long string of the form~\eqref{eq:001} gives $\widetilde{\mathbf{c}} = 0010010001$, which corresponds to the composition $\widetilde\alpha=(3,3,4,1)\models 11$.
Adding $k-1$ to the last part of $\widetilde\alpha$ gives a composition $(3,3,4,3)\models 13$ with no part less than $k$. 
\end{remark}

\section{Proof of Theorem~\ref{thm:formula}}\label{sec:formula}

In this section we establish Theorem~\ref{thm:formula}, which gives two closed formulas for the number appearing in Theorem~\ref{thm:comp3}, that is, the number $a^{(m)}_{k,n}$ of partitions of $n$ with exactly $m$ parts not congruent to $1$ modulo $k$, each greater than $k$ for $m,n\ge0$ and $k\ge2$.
In particular, we have $a^{(0)}_{k,0}=1$ and $a^{(m)}_{k,0}=0$ for all $m\ge1$.
To obtain a closed formula for the number $a^{(m)}_{k,n}$, we first derive a formula for its generating function
\[ A_k(x,y) := \sum_{m,n\ge0} a^{(m)}_{k,n} x^n y^m \]
whose specialization $A_k(x,0)=A_k(x)$ is already determined by Proposition~\ref{prop:Ak}.

\begin{proposition}\label{prop:GF}
For $k\ge2$ we have
\[ A_k(x,y) = \frac{1-x^k}{1-x-x^k-(x^{k+2}+x^{k+3}+\cdots+x^{2k})y}. \]
\end{proposition}

\begin{proof}
For $n,m\ge1$, let $\alpha$ be a composition of $n$ with exactly $m$ parts not congruent to $1$ modulo $k$, each greater than $k$.
The first part of $\alpha$ can be written as $ik+j$, where either $i\ge0$ and $j=1$ or $i\ge1$ and $j\in\{2,\ldots,k\}$.
Removing this part from $\alpha$ gives a composition $\alpha'$ of $n-ik-j$, which is counted by either $a^{(m)}_{k,n-ik-j}$ if $j=1$ or $a^{(m-1)}_{k,n-ik-j}$ if $j\ne1$.
Thus we have a recurrence relation
\[ a^{(m)}_{k,n} = \sum_{i\ge0} a^{(m)}_{k,n-ik-1} + \sum_{2\le j\le k} \sum_{i\ge1} a^{(m-1)}_{k,n-ik-j} \quad \text{for all } m,n\ge1. \]
where we set $a^{(m)}_{k,n}:=0$ for $n<0$.
This implies
\begin{align*}
A_k(x,y) &= A_k(x) + \sum_{n\ge1}\sum_{m\ge1} \left( \sum_{i\ge0} a^{(m)}_{k,n-ik-1} + \sum_{2\le j\le k} \sum_{i\ge1} a^{(m-1)}_{k,n-ik-j} \right) x^n y^m \\
&= A_k(x)+ \sum_{i\ge0}\sum_{n\ge ik+1}\sum_{m\ge1}  a^{(m)}_{k,n-ik-1} x^n y^m \\
& \qquad \qquad + \sum_{2\le j\le k} \sum_{i\ge1} \sum_{n\ge ik+j} \sum_{m\ge1} a^{(m-1)}_{k,n-ik-j} x^n y^m   \\
&= A_k(x)+\sum_{i\ge0} x^{ik+1} \left(A_k(x,y)-A_k(x)\right) + \sum_{2\le j\le k} \sum_{i\ge1} x^{ik+j} y A_k(x,y) \\
&= A_k(x) \left( 1-\frac{x}{1-x^k} \right) \\
& \qquad \qquad + A_k(x,y) \left( \frac{x}{1-x^k} + \frac{(x^{k+2}+x^{k+3}+\cdots+x^{2k})y}{1-x^k}  \right).
\end{align*}
From this and the formula for $A_k(x)$ given in Proposition~\ref{prop:Ak}, we derive
\begin{align*}
& A_k(x,y) \left(1-\frac{x}{1-x^k}- \frac{(x^{k+2}+x^{k+3}+\cdots+x^{2k})y}{1-x^k} \right) \\
= &A_k(x) \left(1-\frac{x}{1-x^k} \right) = 1.
\end{align*}
The result follows immediately.
\end{proof}

Now we are ready to prove Theorem~\ref{thm:formula} in two different ways.

\begingroup
\def\thetheorem{\ref{thm:formula}}
\begin{theorem}
For $m,n\ge0$ and $k\ge2$ we have
\begin{align*}
a_{k,n}^{(m)} 
&= \sum_{ \substack{ \lambda\subseteq (k-2)^m \\ i+(k+1)m+jk+ |\lambda| = n }} \binom{i}{m} \binom{i+j-1}{j} m_\lambda(1^m) \\
&= \sum_{i+(k+1)m+jk+\ell(k-1)+h=n} (-1)^\ell \binom{i}{m} \binom{i+j-1}{j} \binom{m}{\ell} \binom{m+h-1}{h}.
\end{align*}
\end{theorem}
\addtocounter{theorem}{-1}
\endgroup

\begin{proof}[Proof 1]
We first present the generating function proof.
By Proposition~\ref{prop:GF}, we have 
\begin{align*}
A_k(x,y) &= \frac{1-x^k}{1-x^k-\left( x+(x^{k+2}+x^{k+3}+\cdots+x^{2k})y \right) } \\
&= \frac{1}{1-\left(1-x^k \right)^{-1} \left( x+(x^{k+2}+x^{k+3}+\cdots+x^{2k})y \right) } \\
&= \sum_{i\ge0} \left(1-x^k \right)^{-i} \left(x+( x^{k+2}+x^{k+3}+\cdots+x^{2k} ) y\right)^i \\
&= \sum_{m\ge0} \sum_{i\ge m} \binom{i}{m} \left(1-x^k \right)^{-i} x^{i-m} \left( x^{k+2}+x^{k+3}+\cdots+x^{2k} \right)^m y^m \\
&= \sum_{m\ge0} \sum_{i\ge m} \binom{i}{m} \left(1-x^k \right)^{-i} x^{i+(k+1)m} \left( 1+x+\cdots+x^{k-2} \right)^m y^m.
\end{align*}
Using the equations~\eqref{eq:NegPower} and~\eqref{eq:monomial}, we can extract the coefficient of $x^n y^m$ and get the first desired formula.
We can also rewrite 
\[ \left( 1+x+\cdots+x^{k-2} \right)^m = \frac{\left(1-x^{k-1} \right)^m}{(1-x)^m}. \]
Applying the binomial theorem and the equation~\eqref{eq:NegPower} to this gives the second desired formula.
\end{proof}

\begin{proof}[Proof 2]
Now we give a combinatorial proof for the two formulas of the number $a_{k,n}^{(m)}$.
By the proof of Theorem~\ref{thm:comp3}, this number enumerates binary sequences of length $n-1$ in which all but $m$ maximal substrings of zeros have length divisible by $k$.
Such a sequence can be written as 
\[ 0^{\alpha_1}1 0^{\alpha_2}1 \cdots 0^{\alpha_{i-1}} 1 0^{\alpha_i} \]
for some integer $i\ge m$.
There are $\binom{i}{m}$ ways to specify the $m$-subset 
\[ R:=\{r: k \nmid \alpha_r\} \subseteq\{1,2,\ldots,i\}. \]
For each $r\in R$, since $k\nmid \alpha_r\ge k$, there exist integers $a_r\ge0$ and $b_r\in\{0,1,\ldots,k-2\}$ such that
\[ \alpha_r = k+1 + a_r k + b_r. \]
For each $s\in\{1,2,\ldots,i\}\setminus R$, there exists an integer $a_s\ge0$ such that $\alpha_s = a_sk$.
The number of ways to choose the nonnegative integers $a_1,\ldots,a_i$ is $\binom{i+j-1}{j}$, where $j:= a_1+\cdots+a_i$.
The number of ways to choose the integers $b_r\in\{0,1,\ldots,k-2\}$ for all $r\in R$ is 
\[ \sum_{ \substack{ \lambda\subseteq(k-2)^m \\ i+(k+1)m+jk+|\lambda|=n } } m_\lambda(1^m) .\]
We can also write this as 
\[ \sum_{i+(k+1)m+jk+\ell(k-1)+h=n} (-1)^\ell \binom{m}{\ell} \binom{m+h-1}{h} \]
by applying inclusion-exclusion to the integer sequences $(b_r\ge0:r\in R)$ with $b_r\ge k-1$ for all $r$ in a prescribed $\ell$-subset $R'\subseteq R$.
\end{proof}

\section{Analogues of Beck's conjectures}\label{sec:Beck}

In this section we review some conjectures of Beck on partitions with restricted parts and provide analogues for compositions.

Let $a(n)$ be the number of partitions of $n$ with exactly one (possibly repeated) even part.
Let $b(n)$ be the difference between the number of parts in all partitions of $n$ into odd parts and the number of parts in all partitions of $n$ into distinct parts.
Let $c(n)$ be the number of partitions of $n$ in which exactly one part is repeated.
Beck~\cite[A090867]{OEIS} conjectured that $a(n)=b(n)$.
Andrews~\cite[Theorem 1]{Andrews} analytically proved that $a(n)=b(n)=c(n)$.
Fu and Tang~\cite[Theorem 1.5]{FuTang} extended the result of Andrews with an analytic proof.
Using Glaisher's bijection, Yang~\cite[Theorem 1.5]{Yang} generalized the above conjecture of Beck to the following result
\[ \# \mathcal O_{1,k}(n) = \frac{1}{k-1} \left( \sum_{\lambda\in\mathcal O_k(n)} \ell(\lambda) - \sum_{\lambda\in\mathcal D_k(n)} \ell(\lambda) \right) \]
where $\mathcal O_{1,k}(n)$, $\mathcal O_k(n)$, and $\mathcal D_k(n)$ are the sets of partitions of $n$ with exactly one (possibly repeated) part divisible by $k$, no part divisible by $k$, or no part occuring at least $k$ times, respectively.

Let $a_1(n)$ be the number of partitions of $n$ in which exactly one part occurs three times and each other part occurs only once.
Let $b_1(n)$ be the difference between the number of parts in all partitions of $n$ into distinct parts and the number of distinct parts in all partitions of $n$ into odd parts.
Beck~\cite[A090867]{OEIS} conjectured that $a_1(n) = b_1(n)$ and Andrews~\cite[Theorem 2]{Andrews} analytically proved this conjecture.
Extending Glaisher's bijection, Yang~\cite[Theorem 1.7]{Yang} proved a more general result 
\[ \# \mathcal T_k(n) = \sum_{\lambda\in\mathcal D_k(n)} \overline\ell(\lambda) - \sum_{\lambda\in\mathcal O_k(n)} \overline\ell(\lambda). \]
Here $\mathcal T_k(n)$ is the set of partitions of $n$ with one part occurring more than $k$ and less than $2k$ times and every other part occuring less than $k$ times, and $\overline\ell(\lambda)$ is the number of distinct parts of $\lambda$.

A partition $\lambda=(\lambda_1,\ldots,\lambda_\ell)$ is said to be \emph{gap-free} if $0\le \lambda_i-\lambda_{i+1}\le 1$ for all $i=1,2,\ldots,\ell-1$.
Let $a_2(n)$ be the number of gap-free partitions of $n$.
Let $b_2(n)$ be the sum of the smallest parts of all partitions of $n$ into an odd number of distinct parts.
Beck~\cite[A034296]{OEIS} conjectured that $a_2(n)=b_2(n)$.
Chern~\cite{Chern} analytically proved this conjecture based on work of Andrews~\cite{Andrews}.
Yang~\cite{Yang} combinatorially proved this conjecture and found connections with work of Wang, Fokkink, and Fokkink~\cite{WFF}.

Now that we have composition analogues for the partition theorems of Euler, Glaisher, and Franklin, it is natural to explore analogues of the above conjectures of Beck in the setting of compositions.
This would hopefully lead to some interesting results as well as connections to other work on compositions. 

An example is given by a special case of the number $a_{k,n}^{(m)}$ in Theorem~\ref{thm:comp3} and Theorem~\ref{thm:formula}.
According to OEIS~\cite[A029907]{OEIS}, the number $A_n:=a_{2,n+3}^{(1)}$ satisfies the following properties.
\begin{itemize}
\item 
One has $A_0=0$, $A_1=1$, and $A_{n+2} = A_{n+1}+A_{n}+F_{n+1}$ for $n\ge0$, where $F_n$ is the $n$th Fibonacci number.
\item
For $n\ge1$ one has the following simple closed formulas:
\[ A_n = \frac15((n+4)F_n + 2nF_{n-1}) 
= \sum_{0\le i\le n-1} \sum_{0\le j\le i/2} \binom{n-j-1}{j} .\]
\item
The number $A_n$ equals both the number of compositions of $n+1$ with exactly one even part and the number of parts in all compositions of $n$ with odd parts.
\end{itemize}  
The last statement above motivates the following result, which has an easy combinatorial proof.

\begin{proposition}
Let $k\ge2$, $1\le r\le k-1$, $1\le s\le k-r$, and $n\ge 0$.
Then the number of compositions of $n+s$ with one part congruent to $r+s$ and every other part congruent to $r$ modulo $k$ equals the number of parts in all compositions of $n$ with each part congruent to $r$ modulo $k$.
\end{proposition}

\begin{proof}
Let $\alpha=(\alpha_1,\ldots,\alpha_\ell)$ be a composition of $n+s$ with $\alpha_i\equiv r+s \pmod k$ for some $i$ and $\alpha_j\equiv r \pmod k$ for all $j\ne i$.
Define $\alpha'_i:=\alpha_i-s$ and $\alpha'_j=\alpha_j$ for all $j\ne i$.
One sees that $\alpha'=(\alpha'_1,\ldots,\alpha'_\ell)$ is a composition of $n$ with $\alpha'_j\equiv r \pmod k$ for all $j=1,2,\ldots,\ell$, and we have a distinguished part $\alpha'_i$ of this composition.

Conversely, given a composition $\beta=(\beta_1,\ldots,\beta_\ell)$ of $n$ with $\beta_j\equiv r \pmod k$ for all $j=1,2,\ldots\ell$ and given a distinguished part $\beta_i$ for some $i$, we have a composition $\beta'=(\beta'_1,\ldots,\beta'_\ell)$ of $n+s$ defined by $\beta'_i:=\beta_i+s$ and $\beta'_j:=\beta_j$ for all $j\ne i$.
The composition $\beta'$ satisfies $\beta'_i\equiv r+s \pmod k$ and $\beta'_j\equiv r \pmod k$ for all $j\ne i$.
\end{proof}

Taking $k=2$ and $r=s=1$ in the above proposition gives the following corollary, which can be viewed as a composition analogue for the conjectures of Beck.

\begin{corollary}
For $n\ge0$, the number of compositions of $n+1$ with exactly one even part equals the number of parts in all compositions of $n$ with odd parts.
\end{corollary}

Next, let $B_n$ be the number of parts in all compositions of $n+1$ with parts greater than $1$; see OEIS~\cite[A010049]{OEIS}.
We give another analogue of Beck's conjectures.

\begin{proposition}\label{prop:BeckComp}
The number of parts in all compositions of $n$ with parts greater than $1$ equals the difference between the number of parts in all compositions of $n$ with odd parts and the number of parts in all compositions of $n+1$ with parts greater than $1$.
\end{proposition}

\begin{proof}
Turban~\cite[Equation~(2.12)]{Turban} showed that
\[ B_n = \frac 15 \left( (2n+3)F_n - nF_{n-1} \right) = \frac 15 \left( (n+3)F_n + nF_{n-2} \right). \]
It follows that
\begin{align*}
B_n+B_{n-1} &= \frac15 \left( (n+3)F_n+(n+2)F_{n-1} + nF_{n-2} + (n-1)F_{n-3} \right) \\
&= \frac15 \left( (n+4)F_n + 2nF_{n-1} \right) = A_n. 
\end{align*}
Thus $B_{n-1}=A_n-B_n$.
\end{proof}


There could be of course other possible composition analogues for Beck's conjectures.
For instance, one can define a composition $\alpha=(\alpha_1,\ldots,\alpha_\ell)$ to be \emph{gap-free} if $|\alpha_i-\alpha_{i+1}|\le 1$ for all $i=1,2,\ldots,\ell-1$.
Although included in OEIS~\cite[A034297]{OEIS}, the number of gap-free compositions of $n$ still needs further study, and that may lead to connections to compositions with other kinds of part constraints.

\begin{remark}
Li and Wang~\cite{LiWang} bijectively proved a generalization of Proposition~\ref{prop:BeckComp} from $k=2$ to larger values of $k$, and obtained a new composition analogue of Beck's conjectures.
\end{remark}

\section*{Acknowledgment}

The author uses SageMath to discover and verify the main results in this paper.
He is grateful to Goerge Beck, Toufik Mansour, and Augustine Munagi for pointing out some typos and useful references on compositions with restricted parts.
He also thanks the anonymous referees for helpful comments and suggestions.

\end{document}